\documentclass{article}

\usepackage{amsmath,amssymb,amsthm}
\usepackage{tikz}
\usepackage{color}
\usepackage[toc]{appendix}
\usepackage{graphicx}
\usepackage{fancyhdr}
\usepackage{enumitem}
\usepackage{bbm}
\usepackage{parskip}
\usepackage{float}
\usepackage{chngpage}
\usepackage{calc}
\usepackage{bigints}
\usepackage{array}
\usepackage{booktabs}
\usepackage{rotating}
\usepackage{multirow}
\usepackage{adjustbox}
\usepackage{tabularx}
\usepackage{verbatim}
\usepackage{mathtools}
\usepackage{ragged2e}
\usepackage[makeroom]{cancel}
\usepackage{caption}
\usepackage{hyperref}
\usepackage{caption}
\usepackage{subcaption}
\usepackage{appendix}
\usepackage{pgfplots}
\pgfplotsset{compat=1.16}
\usepackage[english]{babel}
\usepackage{hyphenat}
\usepackage[makeindex]{imakeidx}
\usetikzlibrary{datavisualization}
\usetikzlibrary{matrix}
\usetikzlibrary{datavisualization.formats.functions}

\newtheorem{theorem}{Theorem}[section]

\newtheorem{lemma}[theorem]{Lemma}

\newtheorem{remark}[theorem]{Remark}

\makeatletter
\def\section{\@startsection {section}{1}{\z@}{3.25ex plus 1ex minus
		.2ex}{1.5ex plus .2ex}{\large\bf}}
\def\subsection{\@startsection{subsection}{2}{\z@}{3.25ex plus 1ex minus
		.2ex}{1.5ex plus .2ex}{\normalsize\bf}}
\@addtoreset{equation}{section} 
\makeatother

\title{A new numerical scheme for It\^o stochastic differential equations based on Wick-type Wong-Zakai arguments}

\author{Alberto Lanconelli\thanks{Dipartimento di Scienze Statistiche Paolo Fortunati, Università di Bologna, Bologna, Italy. \textbf{e-mail}: alberto.lanconelli2@unibo.it} \and Berk Tan Perçin\thanks{Dipartimento di Scienze Statistiche Paolo Fortunati, Università di Bologna, Bologna, Italy. \textbf{e-mail}: berktan.percin2@unibo.it}}

\date{\today}

\begin{document}
	
	\maketitle

	\bigskip
	
	\begin{abstract}
	The aim of this note is to propose a novel numerical scheme for drift-less one dimensional stochastic differential equations of It\^o's type driven by standard Brownian motion. Our approximation method is equivalent to the well known Milstein scheme as long as the rate of convergence is concerned, i.e. it is strongly convergent with order one, but has the additional desirable property of being exact for linear diffusion coefficients. Our approach is inspired by Wick-type Wong-Zakai arguments in the sense that we only smooth the white noise through polygonal approximation of the Brownian motion while keep the equation in differential form. A first order Taylor expansion of the diffusion coefficient allows us to solve the resulting equation explicitly and hence to provide an implementable approximation scheme.  
	\end{abstract}
	
	Key words and phrases: Brownian motion, stochastic differential equation, Wick product, Milstein scheme. \\
	
	AMS 2000 classification: 60H10, 65C30, 60H25.\\
	
	\allowdisplaybreaks

\section{Statement of the main result}

Let $\{X(t)\}_{t\in [0,1]}$ be the unique strong solution of the It\^o stochastic differential equation (SDE)
\begin{align}\label{SDE}
\begin{cases}
dX(t)=\sigma(X(t))dB(t),& t\in ]0,1];\\
X(0)=a,&
\end{cases}
\end{align}
where $\{B(t)\}_{t\in [0,1]}$ is a standard one dimensional Brownian motion, $\sigma:\mathbb{R}\to\mathbb{R}$ is a measurable function while $a\in\mathbb{R}$ is a  constant. Assume that:
\begin{itemize}
	\item $\sigma$ and $\sigma\sigma'$ are continuously differentiable with bounded derivatives.
\end{itemize}
We have the following result.

\begin{theorem}\label{main theorem}
For $N\in\mathbb{N}$ let $\{x_N(k)\}_{k\in\{0,...,N\}}$ be defined recursively as
\begin{align}
\begin{cases}\label{scheme}
x_N(k+1)=x_N(k)+\frac{\sigma(x_N(k))}{\sigma'(x_N(k))}\left(\exp^{\diamond}\left\{\sigma'(x_N(k))\left(B\left(\frac{k+1}{N}\right)-B\left(\frac{k}{N}\right)\right)\right\}-1\right),& k\in\{0,...,N-1\};\\
x_N(0)=a.
\end{cases}
\end{align}
Then, for all $N\in\mathbb{N}$ we have
\begin{align}\label{inequality theorem}
\mathbb{E}[|X(1)-x_N(N)|]\leq \frac{\mathtt{c}}{N},
\end{align}
where $\{X(t)\}_{t\in [0,1]}$ is the unique strong solution of the SDE \eqref{SDE} while $\mathtt{c}$ is a positive constant independent of $N$. Moreover, when $\sigma:\mathbb{R}\to\mathbb{R}$ is linear, the sequence $\{x_N(k)\}_{k\in\{0,...,N\}}$ coincides with $\{X\left(\frac{k}{N}\right)\}_{k\in\{0,...,N\}}$. \\
In words: scheme \eqref{scheme} provides a strong approximation of order one for the unique strong solution of the SDE \eqref{SDE}; in addition, the scheme is exact when the diffusion coefficient $\sigma$ is linear.
\end{theorem}

\begin{remark}[Implementability of the scheme]\label{wick exponential}
Scheme \eqref{scheme} is fully explicit since the term 
\begin{align}\label{1}
\exp^{\diamond}\left\{\sigma'(x_N(k))\left(B\left(\frac{k+1}{N}\right)-B\left(\frac{k}{N}\right)\right)\right\},
\end{align}
which denotes a \emph{Wick exponential} (see e.g. \cite{OksendalSPDEbook,Hubook,Janson}), is simply a shorthand notation for 
\begin{align*}
\exp\left\{\sigma'(x_N(k))\left(B\left(\frac{k+1}{N}\right)-B\left(\frac{k}{N}\right)\right)-\frac{\sigma'(x_N(k))^2}{2N}\right\}
\end{align*}
which makes the overall numerical scheme as
\begin{align*}
	x_N(k+1)=x_N(k)+\frac{\sigma(x_N(k))}{\sigma'(x_N(k))}\Bigg(\exp\Bigg\{\sigma'(x_N(k))\Bigg(B\Big(\frac{k+1}{N}\Big)-B\Big(\frac{k}{N}\Big)\Bigg)-\frac{\sigma'(x_N(k))^2}{2N}\Bigg\} -1\Bigg).
\end{align*}
This expression is fully explicit and hence implementable for approximation purposes.
\end{remark}

\begin{remark}[Vanishing derivative of $\sigma$]\label{vanishing}
Possibly vanishing values of $\sigma'(x_N(k))$ in the ratio $\frac{\sigma(x_N(k))}{\sigma'(x_N(k))}$ from \eqref{scheme} do not cause any problem of ill-posedness. In fact, exploiting the Wick-series expansion of \eqref{1} we can write conditionally on $\mathcal{F}^B_{k/N}$ (here we denote by $\{\mathcal{F}^B_t\}_{t\in [0,1]}$ the augmented Brownian filtration) that
\begin{align*}
&\frac{\sigma(x_N(k))}{\sigma'(x_N(k))}\left(\exp^{\diamond}\left\{\sigma'(x_N(k))\left(B\left(\frac{k+1}{N}\right)-B\left(\frac{k}{N}\right)\right)\right\}-1\right)\\
&\quad\quad=\frac{\sigma(x_N(k))}{\sigma'(x_N(k))}\sum_{j\geq 1}\frac{\sigma'(x_N(k))^j}{j!}\left(B\left(\frac{k+1}{N}\right)-B\left(\frac{k}{N}\right)\right)^{\diamond j}\\
&\quad\quad=\sigma(x_N(k))\sum_{j\geq 1}\frac{\sigma'(x_N(k))^{j-1}}{j!}\left(B\left(\frac{k+1}{N}\right)-B\left(\frac{k}{N}\right)\right)^{\diamond j},
\end{align*}
thus making perfectly sense for possibly vanishing values of $\sigma'(x_N(k))$. Here, the symbol $\cdot^{\diamond j}$ denotes the $j$-th Wick power.
\end{remark}

\begin{remark}[Comparison with Milstein scheme]\label{Milstein remark}
	If we Taylor expand the Wick exponential in \eqref{scheme} up to the second order, we recover the famous Milstein scheme \cite{Milstein,KloedenPlaten}. In fact, recalling the Wick-series expansion of \eqref{1} utilized in Remark \ref{vanishing} we can write
\begin{align*}
&\frac{\sigma(x_N(k))}{\sigma'(x_N(k))}\left(\exp^{\diamond}\left\{\sigma'(x_N(k))\left(B\left(\frac{k+1}{N}\right)-B\left(\frac{k}{N}\right)\right)\right\}-1\right)\\
&\approx\frac{\sigma(x_N(k))}{\sigma'(x_N(k))}\left(\sigma'(x_N(k))\left(B\left(\frac{k+1}{N}\right)-B\left(\frac{k}{N}\right)\right)+\frac{\sigma'(x_N(k))^2}{2}\left(B\left(\frac{k+1}{N}\right)-B\left(\frac{k}{N}\right)\right)^{\diamond 2}\right)\\
&=\sigma(x_N(k))\left(B\left(\frac{k+1}{N}\right)-B\left(\frac{k}{N}\right)\right)+\frac{\sigma(x_N(k))\sigma'(x_N(k))}{2}\left(B\left(\frac{k+1}{N}\right)-B\left(\frac{k}{N}\right)\right)^{\diamond 2}\\
&=\sigma(x_N(k))\left(B\left(\frac{k+1}{N}\right)-B\left(\frac{k}{N}\right)\right)+\frac{\sigma(x_N(k))\sigma'(x_N(k))}{2}\left(\left(B\left(\frac{k+1}{N}\right)-B\left(\frac{k}{N}\right)\right)^{2}-\frac{1}{N}\right),
\end{align*}	
where in the last equality we exploit the identity between Wick powers and Hermite polynomials \cite{Janson}. Substituting the last member above into the right hand side of \eqref{scheme} we obtain the Mistein scheme for the SDE \eqref{SDE}.
\end{remark}

\begin{remark}[Equation with drift]
Following the proof of Theorem \eqref{main theorem} presented in Section \ref{proof} below, one can see how the introduction of a drift term in equation \eqref{SDE} causes a non trivial issue in the derivation of our scheme. In fact, while the Wick exponential possesses a Wick-free representation as mentioned in Remark \ref{wick exponential}, the corresponding term for an equation with drift will fail in that direction. And since the scheme must be fully implementable, the presence of general terms involving Wick products may prevent to reach that goal. We devote to a future study the investigation of such an extension.  
\end{remark}

\begin{remark}
Our main theorem is a further example of the beneficial effect of treating some aspects of the theory of It\^o SDEs from a Wick calculus' perspective. We mention in passing that the role of Wick product as \emph{natural} Gaussian convolution, shown in \cite{DaPeloLanconelliStan_Young} and \cite{LanconelliStan_Bernoulli}, has been identified for Poisson \cite{LanconelliStan_PoissonWick} and Gamma \cite{LanconelliSportelli_Gamma} distributions as well.	The approach presented here (which essentially relies on the properties of chaos decompositions) together with those contributions may lead to analogous approximation results for stochastic equations driven by non-Gaussian noises. 
\end{remark}

\subsection{Heuristic derivation of the scheme}

In this section we describe some basic ideas that lead to the derivation of scheme \eqref{scheme}. This is inspired by Wong-Zakai approximation results \cite{WZ1,WZ2} for It\^o's type SDEs firstly obtained in \cite{HuOksendal} (see also \cite{Hubook} and the references quoted there) and then extended to different contexts in \cite{LanconelliBenAmmou,DaPeloLanconelliStan_WZ,Lanconelli_FP,LanconelliScorolli_WZ,LanconelliScorolli_Zakai}. The key observation guiding the approach has to be found in the correspondence between It\^o-Skorohod integration and Gaussian Wick product; for a detailed discussion on this issue we refer the reader to \cite{OksendalSPDEbook}.
Exploiting this fact we can indeed rewrite the Cauchy problem \eqref{SDE} as
\begin{align}\label{SDE2}
\begin{cases}
\dot{X}(t)=\sigma(X(t))\diamond\dot{B}(t),& t\in ]0,1];\\
X(0)=a,&
\end{cases}
\end{align}
where $\{\dot{B}(t)\}_{t\in [0,1]}$ denotes Gaussian White Noise. Now, for $N\in\mathbb{N}$ we consider the partition of the interval $[0,1]$ given by $\{0,1/N,2/N,...,1\}$ and denote by $\{B_N(t)\}_{t\in [0,1]}$ the associated polygonal approximation of $\{B(t)\}_{t\in [0,1]}$; this means that
\begin{align*}
B_N(t):=B(k/N)+(t-k/N)\frac{B((k+1)/N)-B(k/N)}{1/N},\quad t\in \left[k/N,(k+1)/N\right[, k\in \{0,...,N-1\},
\end{align*}
and replacing $\{B(t)\}_{t\in [0,1]}$ with $\{B_N(t)\}_{t\in [0,1]}$ we get
\begin{align}\label{SDE_N}
\begin{cases}
\dot{X}_N(t)=\sigma(X_N(t))\diamond\dot{B}_N(t),& t\in ]0,1];\\
X_0=a.&
\end{cases}
\end{align}
Notice that the function $t\mapsto\dot{B}_N(t)$ is piece-wise constant; therefore, we can rewrite equation \eqref{SDE_N} more accurately as the collection of the following nested Cauchy problems
\begin{align}\label{SDE system}
\begin{cases}
\dot{x}_{N,k+1}(t)=\sigma(x_{N,k+1}(t))\diamond\frac{B((k+1)/N)-B(k/N)}{1/N},& t\in ]k/N,(k+1)/N];\\
x_{N,k+1}(k/N)=x_{N,k}(k/N),&
\end{cases}
\end{align}
for $k\in\{0,...,N-1\}$ and with the initial condition: $x_{N,1}(0)=x_{N,0}(0)=a$. We now Taylor-expand up to the first order the function $\sigma$ in \eqref{SDE system} around the initial condition $x_{N,k}(k/N)$ to get
\begin{align}\label{SDE system 2}
\dot{x}_{N,k+1}(t)\approx&\left(\sigma(x_{N,k}(k/N))+\sigma'(x_{N,k}(k/N))(x_{N,k+1}(t)-x_{N,k}(k/N))\right)\diamond\frac{B((k+1)/N)-B(k/N)}{1/N}\nonumber\\
=&\sigma'(x_{N,k}(k/N))x_{N,k+1}(t)\diamond\frac{B((k+1)/N)-B(k/N)}{1/N}\nonumber\\
&+\left(\sigma(x_{N,k}(k/N))-\sigma'(x_{N,k}(k/N))x_{N,k}(k/N))\right)\diamond\frac{B((k+1)/N)-B(k/N)}{1/N}
\end{align}
Equation \eqref{SDE system 2} is linear in $\dot{x}_{N,k+1}(t)$ and can be easily solved either using the reduction method proposed in \cite{HuOksendal} or by means of (conditional on $\mathcal{F}^B_{k/N}$) $\mathcal{S}$-transform \cite{Kuobook} which will convert the previous stochastic equation into a linear deterministic one. In particular, one will find that the value of the solution to \eqref{SDE system 2} at the node $(k+1)/N$ is given by 
\begin{align*}
x_{N,k+1}((k+1)/N)=&x_{N,k}(k/N)\\
&+\frac{\sigma(x_{N,k}(k/N))}{\sigma'(x_{N,k}(k/N))}\left(\exp^{\diamond}\left\{\sigma'(x_{N,k}(k/N))\left(B\left(\frac{k+1}{N}\right)-B\left(\frac{k}{N}\right)\right)\right\}-1\right).
\end{align*}
If we now set $x_N(k):=x_{N,k}(k/N)$ for $k\in\{0,...,N\}$ we see that the previous identity coincides with scheme \eqref{scheme}.
 
\section{Proof of the main result}\label{proof}

To ease the notation, for $k\in\{1,...,N\}$ we set
\begin{align*}
Z_{k}:=\sqrt{N}\left(B\left(\frac{k}{N}\right)-B\left(\frac{k-1}{N}\right)\right),
\end{align*}
so that $\{Z_k\}_{k\in\{1,...,N\}}$ is a collection of independent standard Gaussian random variables; therefore, scheme \eqref{scheme} now reads
\begin{align}
\begin{cases}\label{scheme 2}
x_N(k+1)=x_N(k)+\frac{\sigma(x_N(k))}{\sigma'(x_N(k))}\left(\exp^{\diamond}\left\{\sigma'(x_N(k))Z_{k+1}/\sqrt{N}\right\}-1\right),& k\in\{0,...,N-1\};\\
x_N(0)=a.
\end{cases}
\end{align}
We start proving the following upper bound for the second moment of $x_N(k+1)$. In the sequel we will write $\mathcal{F}_k$ for the sigma algebra generated by $\{Z_1,...,Z_k\}$.
 
\begin{lemma}\label{lemma second moment}
For all $N\in\mathbb{N}$ and $k\in\{0,...,N\}$ we have
\begin{align}\label{lemma inequality}
\mathbb{E}[x_N(k)^2]&\leq \left(1+M_1\frac{\exp\{L_1^2/N\}}{N}\right)^k(a^2+1)-1,
\end{align} 
where $L_1$ stands for the Lipschitz constant of the function $\sigma$ while $M_1$ is a positive constant with which the inequality
\begin{align}\label{M_1}
\sigma(x)^2&\leq M_1(1+x^2),\quad x\in\mathbb{R}
\end{align}
holds true (the existence of such $M_1$ follows from the global Lipschitz continuity of $\sigma$).
\end{lemma}

\begin{proof}
Notice that for all $k\in\{1,...,N\}$ the random variable $x_N(k)$ is $\mathcal{F}_k$-measurable and hence taking the square of both sides in \eqref{scheme 2} and computing conditional expectations with respect to the sigma-algebra $\mathcal{F}_k$ we get
\begin{align*}
\mathbb{E}[x_N(k+1)^2|\mathcal{F}_k]=x_N(k)^2+\frac{\sigma(x_N(k))^2}{\sigma'(x_N(k))^2}\mathbb{E}\left[\left(\exp^{\diamond}\left\{\sigma'(x_N(k))Z_{k+1}/\sqrt{N}\right\}-1\right)^2\bigg|\mathcal{F}_k\right].
\end{align*}
Here, the mixed product of $\frac{\sigma(x_N(k))}{\sigma'(x_N(k))}$ with $\exp^{\diamond}\left\{\sigma'(x_N(k))Z_{k+1}/\sqrt{N}\right\}-1$ vanishes since the first term is $\mathcal{F}_k$-measurable while the second one has null conditional expectation. Moreover, a simple computation gives
\begin{align*}
\mathbb{E}\left[\left(\exp^{\diamond}\left\{\sigma'(x_N(k))Z_{k+1}/\sqrt{N}\right\}-1\right)^2\bigg|\mathcal{F}_k\right]=\exp\left\{\sigma'(x_N(k))^2/N\right\}-1;
\end{align*}
hence,
\begin{align*}
\mathbb{E}[x_N(k+1)^2|\mathcal{F}_k]=x_N(k)^2+\frac{\sigma(x_N(k))^2}{\sigma'(x_N(k))^2}\left(\exp\left\{\sigma'(x_N(k))^2/N\right\}-1\right),
\end{align*}
and taking expectation of both sides we get
\begin{align*}
\mathbb{E}[x_N(k+1)^2]=\mathbb{E}[x_N(k)^2]+\mathbb{E}\left[\sigma(x_N(k))^2\frac{\exp\left\{\sigma'(x_N(k))^2/N\right\}-1}{\sigma'(x_N(k))^2}\right].
\end{align*}
Moreover,
\begin{align*}
\frac{\exp\left\{\sigma'(x_N(k))^2/N\right\}-1}{\sigma'(x_N(k))^2}&=\sum_{j\geq 1}\frac{\sigma'(x_N(k))^{2(j-1)}}{j!N^j}\\
&=\frac{1}{N}\sum_{j\geq 1}\frac{\sigma'(x_N(k))^{2(j-1)}}{j!N^{j-1}}\\
&\leq \frac{1}{N}\sum_{j\geq 1}\frac{\sigma'(x_N(k))^{2(j-1)}}{(j-1)!N^{j-1}}\\
&=\frac{\exp\{\sigma'(x_N(k))^2/N\}}{N}\\
&\leq \frac{\exp\{L_1^2/N\}}{N}.
\end{align*}
Therefore,
\begin{align*}
\mathbb{E}[x_N(k+1)^2]&= \mathbb{E}[x_N(k)^2]+\mathbb{E}\left[\sigma(x_N(k))^2\frac{\exp\left\{\sigma'(x_N(k))^2/N\right\}-1}{\sigma'(x_N(k))^2}\right]\\
&\leq\mathbb{E}[x_N(k)^2]+ \frac{\exp\{L_1^2/N\}}{N}\mathbb{E}\left[\sigma(x_N(k))^2\right]\\
&\leq\mathbb{E}[x_N(k)^2]+ \frac{\exp\{L_1^2/N\}}{N}\mathbb{E}\left[M_1(1+x_N(k)^2)\right]\\
&=\mathbb{E}[x_N(k)^2]\left(1+M_1\frac{\exp\{L_1^2/N\}}{N}\right)+ M_1\frac{\exp\{L_1^2/N\}}{N};
\end{align*}
here, we utilized the inequality
\begin{align*}
\sigma(x)^2\leq M_1(1+x^2),\quad x\in\mathbb{R}.
\end{align*}
The previous computation shows that the sequence of second moments $k\mapsto \mathbb{E}[x_N(k)^2]$ satisfies the recursive inequality
\begin{align*}
\mathbb{E}[x_N(k+1)^2]&\leq\mathbb{E}[x_N(k)^2]\left(1+M_1\frac{\exp\{L_1^2/N\}}{N}\right)+ M_1\frac{\exp\{L_1^2/N\}}{N},
\end{align*} 
with $\mathbb{E}[x_N(0)^2]=a^2$, which gives via a simple algebraic manipulation
\begin{align*}
\mathbb{E}[x_N(k)^2]&\leq a^2\left(1+M_1\frac{\exp\{L_1^2/N\}}{N}\right)^k+M_1\frac{\exp\{L_1^2/N\}}{N}\frac{\left(1+M_1\frac{\exp\{L_1^2/N\}}{N}\right)^{k}-1}{\left(1+M_1\frac{\exp\{L_1^2/N\}}{N}\right)-1}\\
&=a^2\left(1+M_1\frac{\exp\{L_1^2/N\}}{N}\right)^k+\left(1+M_1\frac{\exp\{L_1^2/N\}}{N}\right)^{k}-1\\
&=\left(1+M_1\frac{\exp\{L_1^2/N\}}{N}\right)^k(a^2+1)-1.
\end{align*} 
The proof is complete.
\end{proof}

We are now ready to prove Theorem \ref{main theorem}; we will compare our scheme with the Milstein's one and show that they are of the same order. In the sequel we will denote the Milstein scheme with $\{y_N(k)\}_{k\in\{0,...,N\}}$, namely
\begin{align}
\begin{cases}\label{Milstein scheme}
y_N(k+1)=y_N(k)+\sigma(y_N(k))Z_{k+1}/\sqrt{N}+\sigma(y_N(k))\sigma'(y_N(k))\left(Z_{k+1}^2-1\right)/2N\\
y_N(0)=a.
\end{cases}
\end{align}
We start with 
\begin{align*}
\mathbb{E}[|X(1)-x_N(N)|]&\leq \mathbb{E}[|X(1)-y_N(N)|]+\mathbb{E}[|y_N(N)-x_N(N)|]\\
&\leq \mathbb{E}[|X(1)-y_N(N)|]+\mathbb{E}[|y_N(N)-x_N(N)|^2]^{\frac{1}{2}}.
\end{align*}
It is well known (see \cite{KloedenPlaten}) that there exists a constant $C_1$ such that
\begin{align}\label{z}
\mathbb{E}[|X(1)-y_N(N)|]\leq\frac{C_1}{N};
\end{align}
our aim is to prove that there exists a constant $C_2$ such that
\begin{align*}
\mathbb{E}[|y_N(N)-x_N(N)|^2]^{\frac{1}{2}}\leq\frac{C_2}{N},
\end{align*}
which in combination with \eqref{z} is equivalent to \eqref{inequality theorem}. \\
For all $N\in\mathbb{N}$ and $k\in\{0,...,N-1\}$ we have (recall Remark \ref{Milstein remark})
\begin{align*}
x_N(k+1)-y_N(k+1)=&x_N(k)-y_N(k)+(\sigma(x_N(k))-\sigma(y_N(k)))Z_{k+1}/\sqrt{N}\\
&+(\sigma(x_N(k))\sigma'(x_N(k))-\sigma(y_N(k))\sigma'(y_N(k)))\left(Z_{k+1}^2-1\right)/2N\\
&+\frac{\sigma(x_N(k))}{\sigma'(x_N(k))}\sum_{j\geq 3}\frac{\sigma'(x_N(k))^j}{j!}Z_{k+1}^{\diamond j}/N^{\frac{j}{2}}.
\end{align*}
We now take the square of both sides in the equality above and compute conditional expectations with respect to the sigma-algebra $\mathcal{F}_k$; recall that the Wick powers $Z_{k+1}^{\diamond j}$ are orthogonal in $\mathbb{L}^2(\Omega)$ for different $j$'s (see \cite{Janson}) and hence all the mixed products generated with the square of the right hand side above will vanish in conditional expectation. Moreover, the squared $\mathbb{L}^2(\Omega)$-norm of $Z_{k+1}^{\diamond j}$ is $j!$. Therefore,
\begin{align*}
\mathbb{E}[(x_N(k+1)-y_N(k+1))^2|\mathcal{F}_k]=&(x_N(k)-y_N(k))^2+(\sigma(x_N(k))-\sigma(y_N(k)))^2/N\\
&+(\sigma(x_N(k))\sigma'(x_N(k))-\sigma(y_N(k))\sigma'(y_N(k)))^2/2N^2\\
&+\frac{\sigma(x_N(k))^2}{\sigma'(x_N(k))^2}\sum_{j\geq 3}\frac{\sigma'(x_N(k))^{2j}}{j!^2}j!/N^{j}\\
=&(x_N(k)-y_N(k))^2+(\sigma(x_N(k))-\sigma(y_N(k)))^2/N\\
&+(\sigma(x_N(k))\sigma'(x_N(k))-\sigma(y_N(k))\sigma'(y_N(k)))^2/2N^2\\
&+\frac{\sigma(x_N(k))^2}{\sigma'(x_N(k))^2}\sum_{j\geq 3}\frac{\sigma'(x_N(k))^{2j}}{j!N^j}.
\end{align*}
Let us focus on the last term above:
\begin{align*}
\frac{\sigma(x_N(k))^2}{\sigma'(x_N(k))^2}\sum_{j\geq 3}\frac{\sigma'(x_N(k))^{2j}}{j!N^j}&=\frac{\sigma(x_N(k))^2}{N}\sum_{j\geq 3}\frac{\sigma'(x_N(k))^{2(j-1)}}{j!N^{j-1}}\\
&\leq\frac{\sigma(x_N(k))^2}{N}\sum_{j\geq 3}\frac{\sigma'(x_N(k))^{2(j-1)}}{(j-1)!N^{j-1}}\\
&\leq\frac{\sigma(x_N(k))^2}{N}\sum_{j\geq 3}\frac{L_1^{2(j-1)}}{(j-1)!N^{j-1}}\\
&=\frac{\sigma(x_N(k))^2}{N}\left(e^{L_1^2/N}-1-L_1^2/N\right).
\end{align*}
This gives
\begin{align*}
\mathbb{E}[(x_N(k+1)-y_N(k+1))^2|\mathcal{F}_k]\leq&(x_N(k)-y_N(k))^2+(\sigma(x_N(k))-\sigma(y_N(k)))^2/N\\
&+(\sigma(x_N(k))\sigma'(x_N(k))-\sigma(y_N(k))\sigma'(y_N(k)))^2/2N^2\\
&+\frac{\sigma(x_N(k))^2}{N}\left(e^{L_1^2/N}-1-L_1^2/N\right)
\end{align*}
and taking the expectation of both sides yields
\begin{align}\label{w}
\mathbb{E}[(x_N(k+1)-y_N(k+1))^2]\leq&\mathbb{E}[(x_N(k)-y_N(k))^2]+\mathbb{E}[(\sigma(x_N(k))-\sigma(y_N(k)))^2]/N\nonumber\\
&+\mathbb{E}[(\sigma(x_N(k))\sigma'(x_N(k))-\sigma(y_N(k))\sigma'(y_N(k)))^2]/2N^2\nonumber\\
&+\frac{\mathbb{E}[\sigma(x_N(k))^2]}{N}\left(e^{L_1^2/N}-1-L_1^2/N\right)\nonumber\\
\leq&\left(1+\frac{L_1^2}{N}+\frac{L_2^2}{2N^2}\right)\mathbb{E}[(x_N(k)-y_N(k))^2]\nonumber\\
&+\frac{\mathbb{E}[\sigma(x_N(k))^2]}{N}\left(e^{L_1^2/N}-1-L_1^2/N\right)\nonumber\\
\leq&\left(1+\frac{L_1^2}{N}+\frac{L_2^2}{2N^2}\right)\mathbb{E}[(x_N(k)-y_N(k))^2]\nonumber\\
&+\frac{M_1}{N}\left(e^{L_1^2/N}-1-L_1^2/N\right)\nonumber\\
&+\frac{M_1\left(\left(1+M_1\frac{\exp\{L_1^2/N\}}{N}\right)^N(a^2+1)-1\right)}{N}\left(e^{L_1^2/N}-1-L_1^2/N\right)\nonumber\\
=&\left(1+\frac{L_1^2}{N}+\frac{L_2^2}{2N^2}\right)\mathbb{E}[(x_N(k)-y_N(k))^2]+\Lambda_N.
\end{align}
Here, in the second inequality we exploited the Lipschitz continuity of $\sigma$ (with constant $L_1$) and of $\sigma\sigma'$ (with constant $L_2$); in the third inequality we utilized the upper bound \eqref{M_1} together with Lemma \eqref{lemma second moment} (with $k$ replaced by $N$ in the right hand side of \eqref{lemma inequality}). In addition, we set
\begin{align*}
\Lambda_N:=\frac{M_1\left(1+M_1\frac{\exp\{L_1^2/N\}}{N}\right)^N(a^2+1)}{N}\left(e^{L_1^2/N}-1-L_1^2/N\right).
\end{align*}
Notice that $\Lambda_N\thicksim \frac{1}{N^3}$, as $N$ tends to infinity. If we now compare the first and last members in \eqref{w} we deduce that
\begin{align*}
\mathbb{E}[(x_N(k+1)-y_N(k+1))^2]\leq\left(1+\frac{L_1^2}{N}+\frac{L_2^2}{2N^2}\right)\mathbb{E}[(x_N(k)-y_N(k))^2]+\Lambda_N
\end{align*}
and hence
\begin{align}\label{last}
\mathbb{E}[(x_N(N)-y_N(N))^2]\leq \Lambda_N\frac{\left(1+\frac{L_1^2}{N}+\frac{L_2^2}{2N^2}\right)^N-1}{\frac{L_1^2}{N}+\frac{L_2^2}{2N^2}};
\end{align}
recall that $\mathbb{E}[(x_N(0)-y_N(0))^2]=0$. Observing that the ratio in \eqref{last} goes like $N$ as $N$ tends to infinity we conclude saying that 
\begin{align*}
\mathbb{E}[(x_N(N)-y_N(N))^2]\thicksim \frac{1}{N^2}
\end{align*}
and hence 
\begin{align*}
\mathbb{E}[(x_N(N)-y_N(N))^2]^{\frac{1}{2}}\thicksim \frac{1}{N}
\end{align*}
which is precisely what we wanted to show.

To prove the second part of the statement of Theorem \ref{main theorem} we now assume $\sigma$ to be of the form $\sigma(x)=\alpha x$, $x\in\mathbb{R}$, for some real constant $\alpha\neq 0$. Then, scheme \eqref{scheme} simplifies to
\begin{align*}
x_N(k+1)&=x_N(k)+x_N(k)\left(\exp^{\diamond}\left\{\alpha\left(B\left(\frac{k+1}{N}\right)-B\left(\frac{k}{N}\right)\right)\right\}-1\right)\\
&=x_N(k)\exp^{\diamond}\left\{\alpha\left(B\left(\frac{k+1}{N}\right)-B\left(\frac{k}{N}\right)\right)\right\},\quad k\in\{0,...,N-1\}.
\end{align*}
It is easy to see using standard Wick calculus \cite{OksendalSPDEbook} and the independence of Brownian increments that the solution to the previous recursive equation is given by
\begin{align*}
x_N(k)&=a\exp^{\diamond}\left\{\alpha B\left(\frac{1}{N}\right)\right\}\cdot\cdot\cdot\exp^{\diamond}\left\{\alpha\left(B\left(\frac{k}{N}\right)-B\left(\frac{k-1}{N}\right)\right)\right\}\\
&=a\exp^{\diamond}\left\{\alpha B\left(\frac{k}{N}\right)\right\},\quad k\in\{0,...,N\}.
\end{align*}
Recalling the correspondence between Wick and ordinary exponentials mentioned in Remark \ref{wick exponential} we find:
\begin{equation}\label{linear diffusion analytic solution}
	x_N(k) = a\exp\left\{ \alpha B\left(\frac{k}{N}\right)-\frac{\alpha^2k}{2N} \right\}
\end{equation}
so the $x_N(k)$ written above in expression \eqref{linear diffusion analytic solution} coincides with $X\left(\frac{k}{N}\right)$, the strong solution of \eqref{SDE}, with $\sigma(x)=\alpha x$, evaluated at $\frac{k}{N}$. The proof is complete.
 
\bibliographystyle{abbrv}
\bibliography{wick_numerical_SDE}

\end{document}